\documentclass[draft,11pt,leqno]{article}
\usepackage{amsmath,amssymb,amsthm}
\usepackage{amsthm,amssymb,amsmath,amsfonts} 
\usepackage{eucal}      
\usepackage{mathrsfs}   
\usepackage{longtable}
\usepackage{multirow}
\usepackage{verbatim}
\usepackage{placeins}
\usepackage{graphicx}
\usepackage[T1]{fontenc}
\usepackage{color}

\usepackage{lineno,hyperref}

\def\Rset{\mathbb{R}}

\theoremstyle{plain}
\newtheorem{thm}{Theorem}[section]

\newtheorem{cor}[thm]{Corollary}

\theoremstyle{definition}
\newtheorem{defn}[thm]{Definition}
\newtheorem{rem}[thm]{Remark}
\newtheorem{exmp}[thm]{Example}

\title{\bf Remarks about the existence of conformable derivatives and some consequences}

\author{Hristo Kiskinov $^{1}$, Milena Petkova $^{2}$, Andrey Zahariev $^{3}$\vspace{0.2cm}\\ 
Faculty of Mathematics and Informatics, University of Plovdiv, \\Plovdiv 4000, Bulgaria\vspace{0.2cm}\\
$^{1}$  kiskinov@uni-plovdiv.bg;
$^{2}$  milenapetkova@uni-plovdiv.bg;\\
$^{3}$  zandrey@uni-plovdiv.bg
}
\date{}

\begin{document}
\maketitle

\begin{abstract}
The aim of the present paper is to make some notes 
to the newly introduced conformable derivative as a type local fractional derivative and to
present a surprising result about 
the relation between the conformable derivatives and the usual integer order derivatives.
	

\end{abstract}

\section{Introduction}
\label{sec:introduction}

The fractional  calculus is now almost 325 year old and attracted many researches in the last and present centuries.
There are a lot of definitions of fractional derivatives with different properties.
For a good introduction on the fractional calculus theory and fractional differential equations 
with the "classical" nonlocal fractional derivatives of Riemann-Liouville and Caputo,
see the monographs of Kilbas et al. \cite{KST06}, Kiryakova \cite{Kir94} and Podlubny \cite{Pod99}.

In 2014, Khalil, Al Horani, Yousef and Sababheh \cite{KHYS14} 
introduced a definition of a local kind derivative called from the authors conformable fractional derivative. 
As an important reason for its introduction is specified the fact 
that this derivative satisfies a big part from the well-known properties of the integer order derivatives. 
In 2015, Abdeljawad \cite{Abd15} made an extensive research of the newly introduced conformable fractional calculus. 
In \cite{Mar18} Martynyuk presented a physical interpretation of the conformable derivative.
In the last years there are published more than hundred research articles using this derivative - see for example 
\cite{AAAJO18}--\cite{AA17}, \cite{BLNS15}--\cite{ER15},\cite{GUC15}--\cite{IN16},\cite{LT17},\cite{MWR19,MER20},\cite{TN16,ZFW15}   
and the references therein.

In \cite{AU15} Anderson and Ulness made a remark, that since the derivative is local, 
the correct name must be "conformable derivative" instead as introduced "conformable fractional derivative".
Maybe this is the good 
reason the authors of 
\cite{Mar18,MS18,MSS19-1,MSS19-2}  
to use the name "fractional-like" instead "conformable" derivative.  
When deciding fractional or fractional-like,
it must be noted that in \cite{OM15} and \cite{OM17} are given very good and reasonable classifications 
under different groups of criteria, which one derivative can be called "fractional" (see also \cite{Giu18}). 
In this direction, we refer also to \cite{OM18, OMFM19, Tar18}.

In this work we analyze the conformable derivative introduced in \cite{KHYS14}. 
We study the relationship between the conformable derivatives of different order.
As an surprising result we obtain that a function has a conformable derivative at a point if and only if it has a first order derivative
at the same point and that holds for all points except the lower terminal.
Some considerations what happens in the lower terminal are given too.
As a consequence we consider a nonlinear differential equation with conformable derivatives 
and present a scheme, how to reduce such an equation to an equivalent equation with ordinary first order derivatives.
Some conclusions also are given.

The paper is organized as follows: 
In Section~\ref{sec:preliminaries}, we give some needed definitions 
for the conformable derivatives. 
In Section~\ref{sec:unexpected-result} 
we discuss some properties of the conformable derivatives and clear the connection between two conformable derivatives of different order. 
As a consequence we obtain as main result that a function has a conformable derivative at a point 
(which does not coincide with the lower terminal of the conformable derivative) 
if and only if it has a first order derivative at the same point.  
In Section~\ref{sec:IVP} 
we state an initial value problem for a nonlinear differential equation with conformable derivatives 
and reduce it to an equivalent initial value problem for an equation with ordinary first order derivatives.
Section~\ref{sec:Conclusions} is devoted to our comments of the obtained results. 

%
%

\section{Preliminaries}
\label{sec:preliminaries}


For convenience and to avoid possible misunderstandings, below we recall the definitions of the 
conformable integral
and the conformable derivative introduced in  \cite{KHYS14} as well as some needed their properties. 
For details and more properties, we refer also \cite{Abd15}. 

Below we will use the 
notations  
${\Rset_+ = (0,\infty)}$ and ${\overline{\Rset}_+ = [0,\infty)}$.

Let denote by ${L_1^{loc}(\Rset,\Rset)}$ 
the linear space of all locally Lebesgue integrable functions ${f \colon \Rset \to \Rset}$. 
Then for each ${t > a}$ and ${f \in L_1^{loc}([a,\infty),\Rset)}$ the left-sided 
conformable integral of order ${\alpha \in (0,1]}$ with lower terminal ${a \in \Rset}$ 
is defined by
\begin{equation}   \label{eq:conformable-integral}
I^{\alpha}_a f(t) = \int_a^t (s - a)^{\alpha - 1}\,f(s)ds
\end{equation}
(see \cite{Abd15}, \cite{KHYS14}).
For example this integral exists if $f$ is locally bounded at $a$.

\begin{defn} [\cite{Abd15}, \cite{KHYS14}]   \label{def:left-sided-conformable-derivatives}
	The left-sided conformable derivative  of order ${\alpha \in (0,1]}$ at the point 
	${t \in (a,\infty)}$ for a function ${f \colon [a,\infty) \to \Rset}$ is defined by 
	\begin{equation}   \label{eq:left-sided-conformable-derivatives}
	T^{\alpha}_a f(t) = \lim_{\theta \to 0}\left(\frac{f(t+\theta(t-a)^{1-\alpha}) - f(t)}{\theta}\right)
	\end{equation}
	if the limit exists. 
\end{defn}

As in the case of the classical fractional derivatives the point ${a \in \Rset}$ appearing in \eqref{eq:left-sided-conformable-derivatives} 
will  be called lower terminal of the left-sided conformable derivative. 
Usually, if for $f$ the conformable derivative of order $\alpha$ exists, then for simplicity we say that $f$ is $\alpha$-differentiable.

It may be noted, that some authors (see for example \cite{Abd15}) 
use the notation $T_{\alpha}^a $  instead $T^{\alpha}_a $, 
but we prefer to follow the traditions from the notations of the classical fractional derivatives 
and will write the lower terminal below and the order above.

\begin{defn} [\cite{Abd15}, \cite{KHYS14}]   \label{def:alpha-derivative-of-f}
	The ${\alpha}$-derivative of $f$ at the lower terminal point $a$ 
	in the case when $f$ is  ${\alpha}$-differentiable 	in some interval ${(a,a+\varepsilon)}$, ${\varepsilon > 0}$ 
	is defined as 
	\[
	T^{\alpha}_a f(a) = \lim_{t \to a_+} T^{\alpha}_a f(t)
	\]
	if the limit ${\displaystyle \lim_{t \to a_+} T^{\alpha}_a f(t)}$ exists.
\end{defn}

\begin{rem} 
Note that all definitions and all statements in \cite{KHYS14} are given and proved in the case when the lower terminal is not less than zero. 
Our definitions below are based on the definitions given in \cite{Abd15} for arbitrary lower terminal ${a \in \Rset}$. 
It is not difficult to see that the correct proofs in \cite{KHYS14} can be slightly modified to be actual 
in the case of arbitrary lower terminal ${a \in \Rset}$ too. 
\end{rem}

\begin{defn}    \label{def:left-right-CD}
	The left (right) left-sided conformable derivative of order $\alpha$ at the point ${t \in (a,\infty)}$ is defined by   
	\begin{equation}
	\begin{split}
	& T^{\alpha}_a f(t-0) = \lim_{\theta \to 0-}\left(\frac{f(t+\theta(t-a)^{1-\alpha}) - f(t)}{\theta}\right) \\
	& \left(T^{\alpha}_a f(t+0) = \lim_{\theta \to 0+}\left(\frac{f(t+\theta(t-a)^{1-\alpha}) - f(t)}{\theta}\right)\right).
	\end{split}
	\end{equation}
\end{defn} 
\noindent 
Obviously $f$ is left-sided $\alpha$-differentiable 
at the point ${t \in (a,\infty)}$ if and only if $f$ is left and right left-sided  $\alpha$-differentiable 
at the point ${t \in (a,\infty)}$ and ${T^{\alpha}_a f(t+0) = T^{\alpha}_a f(t-0)}$.

Let $f$ be right left-sided $\alpha$-differentiable in some interval ${(a,a+\varepsilon)}$, ${\varepsilon > 0}$ 
and the limit ${\displaystyle\lim_{t \to a+}T^{\alpha}_a f(t+0)}$ exists. 

\begin{defn}
	The right left-sided conformable derivative of order $\alpha$ at the lower terminal ${a \in \Rset}$ we define with 
	${\displaystyle T^{\alpha}_a f(a+0) = \lim_{t \to a+}T^{\alpha}_a f(t+0)}$.
\end{defn}

Note that in the case when $f$ is $\alpha$-differentiable in some interval ${(a,a+\varepsilon)}$, ${\varepsilon > 0}$, 
we have that ${T^{\alpha}_a f(t) = T^{\alpha}_a f(t+0)}$ and hence ${T^{\alpha}_a f(a) = T^{\alpha}_a f(a+0)}$, i.e. both definitions coincide. 

If ${a = 0}$ we write ${T^{\alpha} f(t+0) = T^{\alpha}_0 f(t) = T^{\alpha} f(t)}$ as usual. 
If $f$ is $\alpha$-differentiable in some finite or infinite interval 
${J \subset [a,\infty)}$ 
we will write that  ${f \in C^{\alpha}_a (J, \Rset)}$, 
where with the indexes $a$ and $\alpha$ are denoted the lower terminal and the order of the conformable derivative respectively. 

In our exposition below we will use only left-sided conformable derivative 
(all definitions and statements for the right-sided conformable derivatives are mirror analogical) 
and for shortness we will omit the expression "left-sided". 

The proofs of the standard statements listed below 
as Theorem~\ref{thm:f-g} 
can be found in  \cite{Abd15}, \cite{KHYS14}.

\begin{thm} [\cite{Abd15}, \cite{KHYS14}]   \label{thm:f-g}
	Let ${\alpha \in (0,1]}$, $c,d \in \Rset$ and $J \subset (a,\infty)$. We assume that ${f,g \in C^{\alpha}_a (J, \Rset)}$. 
	Then for ${t \in J}$ the following relations hold:
	\begin{enumerate}
		\item [$(i)$] ${T^{\alpha}_a (cf+dg) = c\,T^{\alpha}_a f + d\,T^{\alpha}_a g}$;\vspace{0.1cm}
		\item [$(ii)$] ${T^{\alpha}_a (fg) = g\,T^{\alpha}_a f + f\,T^{\alpha}_a g}$;\vspace{0.1cm}
		\item [$(iii)$] ${T^{\alpha}_a (fg^{-1}) = \left(f\,T^{\alpha}_a g - g\,T^{\alpha}_a f\right)g^{-2}}$;\vspace{0.1cm}
		\item [$(iv)$] ${T^{\alpha}_a (1) = 0}$;\vspace{0.1cm}
		\item [$(v)$] ${T^{\alpha}_a f(t) = (t - a)^{1-\alpha}f'(t)}$ if in addition $f$ is differentiable for ${t \in J}$.
	\end{enumerate}
\end{thm}


%

%
%

\section{The close connection between the conformable and the first order derivatives}
\label{sec:unexpected-result}

Introducing the notion conformable derivative, the authors of \cite{KHYS14} present the following example which must assure the readers 
that a function "could be $\alpha$-differentiable at a point but not differentiable at the same point".

\begin{exmp} [\cite{KHYS14}] As example is taken the function ${\displaystyle f(t) = 2t^{\frac{1}{2}}}$ for ${t \in \overline{\Rset}_+}$. For this function we have that ${\displaystyle T^{\frac{1}{2}}f(t) = 1}$ for every ${t \in \Rset_+}$. Hence ${\displaystyle T^{\frac{1}{2}}f(0) = \lim_{t \to 0+}T^{\frac{1}{2}}f(t) = 1}$. Then $f$ is $\alpha$-differentiable at the point zero, but ${f'(0)}$ does not exist.
\end{exmp}

At first glance it seems ok, 
but this situation is possible only when the left end of the studied interval coincides with the lower terminal of the conformable derivative
(in the example above this is the point ${a = 0}$). 

Let compare with the following example.
\begin{exmp} Let the lower terminal of the conformable derivative be ${a = -1}$ instead ${a = 0}$ and consider as standard prolongation of ${\displaystyle f(t) = 2t^{\frac{1}{2}}}$ to the interval ${t \in [-1,\infty)}$ the function ${\displaystyle \overline{f}(t) = 2|t|^{\frac{1}{2}}}$.
Since the point ${t = 0}$ is in this case an inner point for ${[-1,\infty)}$ and does not coincide with the lower terminal ${a = -1}$ then 
	\[
	T^{\frac{1}{2}}_{-1}\overline{f}(0) = \lim_{\theta \to 0}\frac{\overline{f}(\theta) - \overline{f}(0)}{\theta} = \lim_{\theta \to 0}\frac{2|\theta|^{\frac{1}{2}}}{\theta} = 2sign \theta \lim_{\theta \to 0}|\theta|^{-\frac{1}{2}} .
	\]
But since  $2 sign \theta\lim\limits_{\theta \to \pm 0}|\theta|^{-\frac{1}{2}} = \pm \infty$ 
then ${\displaystyle T^{\frac{1}{2}}_{-1}\overline{f}(0)}$ 
does not exist as the first derivative too. 
\end{exmp}

Comparing both examples it is clear that the only difference between them is that 
in the first of them, the studied for derivatives point is the left end of the considered interval and 
coincides with the lower terminal of the conformable derivative, 
while in the second one, the same point studied for derivatives is an inner point and 
does not coincide with the lower terminal of the conformable derivative. 
This observation leads us to the main result of this work.

\begin{thm}   \label{thm:conformable-derivative-existing}
	Let ${f : [a,\infty) \to \Rset}$ and there exist a point ${t_0 \in (a,\infty)}$ and number ${\alpha \in (0, 1]}$ 
	such that the conformable derivative ${T^{\alpha}_a f(t_0)}$ with lower terminal point $a$ exists.
	
	Then the conformable derivative ${T^{\beta}_a f(t_0)}$ exists 
	for every ${\beta \in (0, 1]}$ with ${\beta \neq \alpha}$ and 
	\[
	{T^{\alpha}_a f(t_0) = (t_0 - a)^{\beta-\alpha}T^{\beta}_a f(t_0)}.
	\]
\end{thm}

\begin{proof}
Let ${\beta \in (0, 1]}$ with ${\beta \neq \alpha}$ be arbitrary. 
Then we have 
\begin{equation}    \label{eq:1}
\begin{split}
T^{\beta}_a f(t_0)& = \lim_{\theta \to 0}\frac{f(t_0 + \theta(t_0 - a)^{1-\beta+\alpha-\alpha}) - f(t_0)}{\theta} \\ 
& = \lim_{\theta \to 0}\frac{f(t_0 + \theta(t_0 - a)^{(1-\alpha)+(\alpha-\beta)}) - f(t_0)}{\theta} \\
& = (t_0 - a)^{\alpha-\beta}\lim_{\theta \to 0}\frac{f(t_0 + \theta(t_0 - a)^{\alpha-\beta}(t_0 - a)^{1-\alpha}) - 
       f(t_0)}{\theta(t_0 - a)^{\alpha-\beta}}.
\end{split}
\end{equation}
	Then for every ${\theta \in \Rset}$ and for every fixed $\alpha , \beta \in (0,1]$  
	there exists a unique ${\theta_{\alpha,\beta} \in \Rset}$, 
	such that ${\theta_{\alpha,\beta} = \theta(t_0 - a)^{\alpha - \beta}}$. 
	Obviously when ${\theta \to 0}$ then ${\theta_{\alpha,\beta} \to 0}$ too.  
	Then from \eqref{eq:1} it follows that
	\[
	T^{\beta}_a f(t_0) = (t_0 - a)^{\alpha-\beta}\lim_{\theta \to 0}\frac{f(t_0 + \theta_{\alpha,\beta}(t_0 - a)^{1-\alpha}) - f(t_0)}{\theta_{\alpha,\beta}} = (t_0 - a)^{\alpha-\beta}T^{\alpha}_a f(t_0).
	\]
\end{proof}

\begin{cor}    \label{cor:conformable-iff-first derivative}
	For a function ${f : [a,\infty) \to \Rset}$ the conformable derivative ${T^{\alpha}_a f(t_0)}$ 
	with lower terminal $a$ at a point ${t_0 \in (a,\infty)}$ for some ${\alpha \in (0,1)}$ exists 
	if and only if 
	the function ${f(t)}$ has first derivative at the point ${t_0 \in (a,\infty)}$ and
\begin{equation} \label{e1}
	{T^{\alpha}_a f(t_0) = (t_0 - a)^{1-\alpha}  f'(t_0)}.
\end{equation}
\end{cor}

\begin{proof}
	To prove sufficiency we apply Theorem~\ref{thm:conformable-derivative-existing} for $\alpha=1$ and $\beta < 1$.
	To prove necessity we apply Theorem~\ref{thm:conformable-derivative-existing} for $\alpha<1$ and $\beta =1$. 
\end{proof}

\begin{rem}
During the publishing process we noticed that Corollary \ref{cor:conformable-iff-first derivative} 
is obtained in \cite{A18} in the particular case for conformable derivatives with lower terminal zero. See also \cite{A19}. 
\end{rem}

The next two corollaries treat the problem of the left and right inverse operator of the conformable derivative.

\begin{cor}    \label{cor:conformable-derivative} 
	Let the following conditions hold:
	\begin{enumerate}
		\item [$1.$] The function ${f : [a,\infty) \to \Rset}$ has at most a first kind (bounded) jump at $a$.
		\item [$2.$] For some ${\alpha \in (0,1)}$ the function ${f \in C_a^{\alpha}((a,\infty),\Rset)}$.
	\end{enumerate}
	
	Then for every ${t \in (a,\infty)}$ we have that ${I^{\alpha}_a\,T^{\alpha}_a f(t) = f(t) - f(a)}$.
\end{cor}

\begin{proof}
	Applying Corollary~\ref{cor:conformable-iff-first derivative} we obtain that $f$ possess first derivative at every point ${t \in (a,\infty)}$ 
	and then the statement follows from \eqref{e1}. 
\end{proof}

\begin{rem}
Note that the statement of Corollary~\ref{cor:conformable-derivative} is presented in \cite{Abd15} (Lemma 2.8) without the condition 1
which guaranties that ${I^{\alpha}_a\,T^{\alpha}_a f(t)}$ exists for ${t \in (a,\infty)}$. 
A simple example demonstrates that this condition is essential. 
Let ${\alpha \in (0,1)}$ and  ${f(t) = \alpha^{-1} (t-a)^{\alpha-\beta}}$, where ${\beta \in (0,1)}$ with ${\beta > \alpha}$. 
Then ${I^{\alpha}_a  T^{\alpha}_a   f(t)}$ is divergent for every ${t \in (a,\infty)}$.  
\end{rem}

\begin{cor}   \label{cor:as-theorem-3-1-with-more-restriction}
	Let ${f \in L_1^{loc}([a,\infty),\Rset)}$ be locally bounded. 
	
	Then ${T^{\alpha}_a I^{\alpha}_a f(t) = f(t)}$ for ${t \in (a,\infty)}$.
\end{cor}

\begin{proof}
	Since ${f \in L_1^{loc}([a,\infty),\Rset)}$ and it is locally bounded, 
	then ${I^{\alpha}_a f(t)  \in AC((a,\infty),\Rset)}$ and then the statement follows for every ${t \in (a,\infty)}$ 
	from Corollary~\ref{cor:conformable-iff-first derivative}.
\end{proof}

\begin{rem} 
	The statement of Corollary~\ref{cor:as-theorem-3-1-with-more-restriction} is proved in \cite{KHYS14} as Theorem 3.1 in the case when the lower terminal of the conformable derivative is zero and under the more restrictive condition ${f \in C(\overline{\Rset}_+,\Rset)}$ .
It must be also noted, that the statement of Theorem 3.1 in \cite{KHYS14} for ${t = a}$ is wrong.	
\end{rem}

\begin{thm}   \label{thm:conformable-derivative-existing-lower-terminal}
	Let ${f : [a,\infty) \to \Rset}$ and there exists a number ${\alpha \in (0, 1]}$ 
	such that the conformable derivative ${T^{\alpha}_a f(a)}$ with lower terminal point $a$ exists.
	
	Then the conformable derivative ${T^{\beta}_a f(a)}$ also exists for every ${\beta \in (0,\alpha)}$ and 
	\[
	T^{\beta}_a f(a) = 0.
	\]
\end{thm}

\begin{proof} 
Since $T_a^\alpha f(a)$ exists, then according Definition \ref{def:alpha-derivative-of-f} 
there exists $\varepsilon \in \Rset_+$  such that 
$f \in C_a^\alpha((a,a+\varepsilon),\Rset)$ and $T_a^\alpha f(a)= \lim\limits_{t \to a+} T_a^\alpha f(t)$. 
Then in virtue of Theorem~\ref{thm:conformable-derivative-existing} 
for every $\beta \in (0,1)$ we have that $f \in C_a^\beta((a,a+\varepsilon),\Rset)$  
and for each $t \in (a,a+\varepsilon)$  the relation $T_a^\beta f(t) = (t-a)^{\alpha - \beta} T_a^\alpha f(t)$  holds. 
Then for $\beta < \alpha$  we have that 
	\[
	 \lim_{t \to a+} T_a^\beta f(t) =   \lim_{t \to a+} (t-a)^{\alpha - \beta} \lim_{t \to a+}  T_a^\alpha f(t) = 
	      \lim_{t \to a+} (t-a)^{\alpha - \beta}  \ \  T_a^\alpha f(a)   = 0.
	\]
\end{proof}

\begin{cor}    \label{cor:first-to-conformable-low-terminal}
	Let a function ${f : [a,\infty) \to \Rset}$ have first derivative in $(a,a+\varepsilon)$ for some $\varepsilon >0$
	and $\lim\limits_{t \to a+} f'(t)$ exists. 
	Then $f$ is $\alpha$-differentiable at $a$ for all $\alpha \in (0,1)$ and
	\[
	  T^{\alpha}_a f(a) = 0.
	\]
\end{cor}

\begin{proof}
 The statement follows immediately from Theorem \ref{thm:conformable-derivative-existing-lower-terminal}, applied for $\alpha = 1$.
\end{proof}

Researching further in this direction, we can present also the following statements.

\begin{thm}    \label{thm:first-to-conformable-low-terminal-1}
	Let a function ${f : [a,\infty) \to \Rset}$ have bounded first derivative in $(a,a+\varepsilon)$ for some $\varepsilon >0$. 
	Then $f$ is $\alpha$-differentiable at $a$ for all $\alpha \in (0,1)$ and
	\[
	  T^{\alpha}_a f(a) = 0.
	\]
\end{thm}

\begin{proof}
Let $\alpha \in (0,1)$ be arbitrary. 
Then in virtue of Corollary \ref{cor:conformable-iff-first derivative} for each $t \in (a,a+\varepsilon)$  
the relation ${T^{\alpha}_a f(t) = (t - a)^{1-\alpha}  f'(t)}$  holds. 
Since $f'(t)$ is bounded on $(a,a+\varepsilon)$, then there exists a constant $C$, such that $|f'(t)| \leq C$  for each $t \in (a,a+\varepsilon)$. 
Then we obtain for every $t \in (a,a+\varepsilon)$ the estimation $|T^{\alpha}_a f(t)| \leq C (t-a)^{1-\alpha}$, 
which implies that $\lim\limits_{t \to a+} |T^{\alpha}_a f(t)|=0$ and hence  $\lim\limits_{t \to a+} T^{\alpha}_a f(t)=T^{\alpha}_a f(a) =0$.
\end{proof}

\begin{cor}    \label{cor:first-to-conformable-low-terminal-1}
	Let for a function ${f : [a,\infty) \to \Rset}$ is fulfilled  $f \in C^1([a,a+\varepsilon),\Rset)$ for some $\varepsilon >0$. 
	Then $f$ is $\alpha$-differentiable at $a$ for all $\alpha \in (0,1)$ and
	\[
	  T^{\alpha}_a f(a) = 0.
	\]
\end{cor}
\begin{proof}
 The statement follows immediately from Theorem \ref{thm:first-to-conformable-low-terminal-1}.
\end{proof}

\begin{rem}
	It is well known that ${C^n(\Rset,\Rset) \subset C^k(\Rset,\Rset)}$ for integer ${n > k}$. 
	From other side, for all classical fractional derivatives the question 
"Why does the existence of one or other form of a fractional derivative $D_a^\alpha f$ 
lead to existence of a derivative $D_a^\beta f$  of the same form for $\beta < \alpha$ ?" 
is deeply studied (see \cite{SKM93}). 
	Then the question what is the relation between ${C^{\alpha}_a((a,\infty),\Rset)}$ and ${C^{\beta}_a((a,\infty),\Rset)}$, 
when ${\alpha, \beta \in (0,1)}$ with ${\alpha \neq \beta}$ is more than reasonable. 
	The answer is a little unexpected - Theorem \ref{thm:conformable-derivative-existing} states that
${C^{\alpha}_a((a,\infty),\Rset) = C^{\beta}_a((a,\infty),\Rset)}$ for all ${\alpha, \beta \in (0,1]}$. 
	In addition Theorem \ref{thm:conformable-derivative-existing-lower-terminal} implies that from the existence of 
$T_a^\alpha f(a)$ it follows the existence of $T_a^\beta f(a)$ for $\beta \in (0,\alpha)$,
i.e. a fractional-like behavior appears only in the lower terminal.
\end{rem}

\begin{rem}
The unexpected from some point of view result stated in Theorem \ref{thm:conformable-derivative-existing}  
and Corollary \ref{cor:conformable-iff-first derivative}
can be explained if we compare the Definition~\ref{def:left-sided-conformable-derivatives} and Definition~\ref{def:alpha-derivative-of-f} of the conformable derivative. 
It is obviously that Definition~\ref{def:left-sided-conformable-derivatives} for the inner points of the considered interval 
is more restrictive in compare with the Definition~\ref{def:alpha-derivative-of-f} for existing of conformable derivative in the point ${a \in \Rset}$, 
which is the lower terminal of the conformable derivative. 
Actually, Definition~\ref{def:alpha-derivative-of-f} means that for the case when $f$ is $\alpha$-differentiable in some interval ${(a,a+\varepsilon)}$, ${\varepsilon > 0}$, 
the existence of $\alpha$-derivative at the lower terminal $a$ of $f$  
requires 
the conformable derivative ${T^{\alpha}_a f(t)}$ to have right limit at the lower terminal $a$. 
After that this limit ${T^{\alpha}_a f(a) = \lim\limits_{t \rightarrow a+}T^{\alpha}_a f(t)}$ 
is called $\alpha$-derivative of $f$ at the lower terminal $a$. 
Thus these differences between both definitions explain why for an inner point 
(which does not coincide with the lower terminal of the conformable derivative) 
the existence of the conformable derivative for some ${\alpha \in (0,1)}$ 
is equivalent to the existence of first derivative in the same point
and why a difference between both derivatives can appear only at the lower terminal point.  
\end{rem}

%
%
\section{Initial Value Problem for nonlinear differential equation with conformable derivative}
\label{sec:IVP}

Consider the initial value problem (IVP) for nonlinear differential equation with conformable derivative                                            
\begin{equation}  \label{eq:nonlinear-system}
T^{\alpha}_a x(t) = F(t,x(t))
\end{equation}
with initial condition
\begin{equation}  \label{eq:initial-condition}
x(a) = x_a \in \Rset,
\end{equation}
where ${a \in \Rset}$, ${x : [a,\infty) \to \Rset}$, ${F : [a,\infty) \times \Rset \to \Rset}$. 
The same consideration holds for systems too,
but for simplicity we will consider an IVP only for equations under the assumption 
that ${F \in C([a,\infty) \times \Rset,\Rset)}$.

Suppose that ${x(t) \in C_a^\alpha((a,\infty),\Rset)}$ and  ${x(t)}$ for every ${t \in (a,\infty)}$ satisfies the equation \eqref{eq:nonlinear-system} and the initial condition \eqref{eq:initial-condition} too. 
Then 
in virtue of Corollary~\ref{cor:conformable-iff-first derivative} we have 
\begin{equation}    \label{eq:substituting-system}
T^{\alpha}_a x(t) = (t-a)^{1-\alpha}x'(t).
\end{equation}
From \eqref{eq:nonlinear-system}, \eqref{eq:initial-condition} and \eqref{eq:substituting-system} we obtain that ${x(t) \in C^1((a,\infty),\Rset)}$.
It is a continuous differentiable solution of the following ordinary differential equation 
which right side includes weak (integrable) singularity at the initial point $a$ 
\begin{equation}    \label{eq:sode}
x'(t) = (t-a)^{\alpha-1}F(t, x(t))
\end{equation}
and satisfies the initial condition \eqref{eq:initial-condition} too. 
Conversely, if ${x(t) \in C^1((a,\infty),\Rset)}$ is a solution of IVP \eqref{eq:sode}, \eqref{eq:initial-condition}, 
then ${x(t)}$ is an $\alpha$-differentiable solution of IVP \eqref{eq:nonlinear-system}, \eqref{eq:initial-condition}. 
It must be noted that the IVP \eqref{eq:sode}, \eqref{eq:initial-condition} is obviously equivalent of the IVP with initial condition \eqref{eq:initial-condition} for the following Volterra integral equation
\begin{equation}   \label{eq:volterra-integral-equations}
x(t) = x_a + \int_a^t(s-a)^{\alpha-1}F(s, x(s))ds
\end{equation}
in the following sense: Every solution ${x(t) \in C([a,\infty),\Rset)}$ of the 
IVP \eqref{eq:volterra-integral-equations}, \eqref{eq:initial-condition} is a solution 
of the IVP \eqref{eq:sode}, \eqref{eq:initial-condition}, (${x(t) \in C^1((a,\infty),\Rset)}$) and vice versa. 

Thus we can conclude that since the considered IVP \eqref{eq:sode}, \eqref{eq:initial-condition} 
and IVP \eqref{eq:volterra-integral-equations}, \eqref{eq:initial-condition} are well studied 
(fundamental theory, standard type of stabilities and etc..) 
then from mathematical point of view the introduced conformable derivatives 
does not provide any real improvement to the theory of fractional calculus
in compare with the classical fractional derivatives. 
Furthermore, they bring noting new at least as mathematical advantages in 
the field of the ordinary differential equations with fractional derivatives.

%
%

\section{Conclusions}
\label{sec:Conclusions}

The existence of a direct one-to-one connection between the conformable derivative and the first order derivative
opens the following question - can a research with conformable derivatives be considered as new 
or is well known old, rewritten in the terms of the conformable calculus?
Our answer is - if the lower terminal does not actively participate in the research object - definitely the second one.

An other question is, can we construct some new valuable mathematical model with conformable derivative for investigation, 
which describes at least one real word phenomena. 
From mathematical point of view may be a consideration of differential equations with conformable derivative and deviating argument 
(delayed or neutral type) can at least compensate the loss of the dependence from the past history of the evolutionary process. 
In this direction it is possible that the study of initial problems for these equations, with different types of initial functions, 
in the case when the lower terminal is the right end of the initial interval, will be not trivial and may be valuable.

%

%
%

\end{document}